\documentclass[sn-aps]{sn-jnl}

\usepackage{graphicx}%
\usepackage{multirow}%
\usepackage{times}
\usepackage{amsmath,amssymb,amsfonts}%
\usepackage{amsthm}%
\usepackage{mathrsfs}%
\usepackage[title]{appendix}%
\usepackage{xcolor}%
\usepackage{textcomp}%
\usepackage{manyfoot}%
\usepackage{booktabs}%
\usepackage{algorithm}%
\usepackage{algorithmicx}%
\usepackage{algpseudocode}%
\usepackage{listings}%

\usepackage{geometry}
\theoremstyle{thmstyleone}
\newtheorem{theorem}{Theorem}[section]
\newtheorem{proposition}[theorem]{Proposition}

\newtheorem{conclusion}{Conclusion}
\newtheorem{remark}[theorem]{Remark}

\theoremstyle{thmstyletwo}
\newtheorem{example}[theorem]{Example}
\theoremstyle{thmstylethree}
\newtheorem{definition}{Definition}[section]
\def\r{\mathbb R}
\def\E{\mathbb E}

\def\t{\mathbf t}
\def\n{\mathbf n}
\def\b{\mathbf b}

\def\x{\mathbf x}
\def\y{\mathbf y}

\def\u{\mathbf u}
\def\v{\mathbf v}

\begin{document}

\title[Multiplicative partner curves in $\E_*^3$]{\Huge A non-Newtonian some partner curves in multiplicative Euclidean space $\E_*^3$}

\author*[1]{\fnm{Aykut} \sur{Has}}\email{ahas@ksu.edu.tr}

\author[2]{\fnm{Beyhan} \sur{Yılmaz}}\email{beyhanyilmaz@ksu.edu.tr}
\equalcont{These authors contributed equally to this work.}


\affil[1,2]{\orgdiv{Department of Mathematics, Faculty of Science}, \orgname{Kahramanmaras Sutcu Imam University}, \orgaddress{ \city{Kahramanmaras}, \postcode{46100}, \country{Turkey}}}




\abstract{The aim of this article is to characterize pairs of curves within multiplicative (non-Newtonian) spaces. Specifically, we investigate how famous curve pairs such as Bertrand partner curves, Mannheim partner curves, which are prominent in differential geometry, are transformed under the influence of multiplicative analysis. By leveraging the relationships between multiplicative Frenet vectors, we introduce multiplicative versions of Bertrand, Mannheim curve pairs. Subsequently, we characterize these curve pairs using multiplicative arguments. Examples are provided, and multiplicative graphs are presented to enhance understanding of the subject matter. Through this analysis, we aim to elucidate the behavior and properties of these curve pairs within the context of multiplicative geometry.}
\maketitle

\section{Introduction}\label{sec1}
A curve can be simply described as the path traced by a point moving in space. As a regular curve traverses its path, it generates a tangent vector field at each point along its trajectory. With the help of this tangent vector field, the principal normal vector field and binormal vector field of the curve are created.  Consequently, these three vector fields collectively form the structure known as the Frenet vectors of the curve \cite{carmo}. Utilizing these Frenet vectors, the curvature and torsion of the curve can be quantified. Categorizing curves based on certain characteristics is a valuable approach in differential geometry and the Frenet apparatus serves as a useful tool for this purpose. For instance, curves whose principal normals are linearly dependent is called Bertrand curve pairs, those whose principal normal and binormal are linearly dependent are referred to as Mannheim curve pairs and curves with orthogonal tangents are known as involute-evolute curve pairs \cite{bertrand,mannheim,boyer}. As can be understood from here, the most important building block of the curve can be said to be its tangent. Because all other concepts can be defined with the help of the tangent of the curve. Singular points, where the curve lacks a defined tangent, have been a subject of recent research interest \cite{hilal,sernesi,takahashi}. Has A. et al. created alternative tangents with the help of conformable derivative at points where the traditional derivative was not defined. Thus, they provided a solution to define an alternative Frenet frame for a curve that does not have a conventional derivative at any point \cite{aykut1,aykut2}. Similarly, Evren M. E. et al. tried to construct differential geometry with the help of multiplicative analysis. In this study, researchers calculated the arc length, which could not be calculated by traditional methods, with multiplicative arguments and solved this problem \cite{evren}. In this respect, it can be said that there are advantages to using different analyzes in differential geometry \cite{aykut3,aykut4,zehra,mert}.

Classical analysis, a mathematical theory widely used today, is discovered by Leibniz G. and Newton I. in the second half of the 17th century, based on the concepts of derivative and integral, and is also referred to as Newtonian analysis. In addition, it can be said that this structure, which uses the unit function as the generator function, forms the basis of all mathematics. This approach measures change by considering infinitesimally small variations in quantities and it employs the unit function as the generator function, forming the cornerstone of mathematical theories. Volterra V. sought to introduce alternative analyses to Newtonian analysis by modifying the generator functions and basic operations (addition, subtraction, multiplication and division) \cite{volterra}. Building on Volterra V. work, Grossman M. and Kantz R. further expanded this framework and introduced various types of analysis such as geometric, bigeometric, and anageometric calculus \cite{grossman,grossman2}. These non-Newtonian or multiplicative analyses use the exponential function as the generator function, where multiplication assumes the role of addition in Newtonian analysis. This allows for changes to be measured not only through the difference operation but also proportionally. Multiplicative analysis is known to offer advantages over traditional analysis, particularly in scenarios involving exponential growth and fractal dimensions, making it increasingly popular in interdisciplinary studies \cite{samuelson,afrouzi,rybaczuk}. In recent years, multiplicative arguments is incorporated into the basic concepts of almost every subject, leading to a redesign of fundamental principles \cite{boruah,bashirov,emrah1,yusuf}. Georgiev S. et al. contributions is instrumental in accelerating this process, particularly in the field of differential geometry \cite{svetlin,svetlin2,svetlin3}. Georgiev S. work are introduced basic definitions and theorems regarding curves, surfaces and manifolds thereby catalyzing advancements in geometry research. With Georgiev S. et al. work as a reference, studies in the field of differential geometry have experienced accelerated progress \cite{karacan,hasan,hasan2,hasan3,aykut5,aykut6}.

By measuring changes proportionally in multiplicative space, all traditionally known metric concepts undergo transformation when employing multiplicative arguments. In this space, numerous metric concepts crucial in the realm of differential geometry, such as angle, length, norm, etc., are redefined based on the properties of the exponential function. The objective of this article is to explore the most prominent curve pairs in the theory of curves, namely Bertrand, Mannheim curve pairs, utilizing new multiplicative metric concepts. Under the influence of multiplicative arguments, these special pairs of curves undergo re-characterization. Furthermore, relevant examples and figures are provided to facilitate a better understanding of the topic.

\section{Multiplicative Calculus and Multiplicative Space} \label{sec2}
In this section, basic definitions and theorems will be given about the multiplicative space created by choosing the generator exponential function $(exp)$. Generator function $\alpha$ is chosen as $(exp)$ function. Georgiev S. books will be used for the basic informations given in this section \cite{svetlin,svetlin2,svetlin3}.
\begin{eqnarray*}
  \alpha&:&\r\rightarrow \r^+\quad \quad \quad \quad \quad \quad \quad \alpha^{-1}:\r^+\rightarrow \r\\
  &&a\rightarrow \alpha(a)=e^a \quad \quad\quad \quad \quad \quad \quad b\rightarrow \alpha^{-1}(b)=\log b.
\end{eqnarray*}
By choosing the generator $(exp)$ function as, a function is defined from real numbers to the positive side of real numbers. Thus, the set of real numbers in the multiplicative space is defined as follows
\begin{equation*}
\r_*=\{exp(a):a\in\r\}=\r^+.  
\end{equation*}
Similarly, positive and negative multiplicative numbers are as follows
\begin{equation*}
\r^{+}_*=\{exp(a):a\in\r^+\}=(1,\infty)  
\end{equation*}
and
\begin{equation*}
\r^{-}_*=\{exp(a):a\in\r^-\}=(0,1).  
\end{equation*}
Additionally, with the help of the function $exp$, the basic operations in the multiplicative space can be seen from Table \ref{n}, for all $a,b\in\r_*$, $b\neq 1$
\begin{equation*}
\begin{tabular}{|l|l|l|}
	\hline
	$a+_*b$  & $e^{\log a+\log b}$ & $ab$ \\
	\hline
	$a-_*b$ & $e^{\log a-\log b}$ & $\frac{a}{b}$ \\
	\hline
	$a\cdot_*b$ & $e^{\log a\log b}$ & $a^{\log b}$ \\
	\hline
	$a/_*b$ & $e^{\log a/\log b}$ & $a^{\frac{1}{\log b}} $  \label{n}\\
	\hline
\end{tabular} 
\end{equation*}
\centerline{\textbf{Table 2.} Basic multiplicative operations.}
Given by Table \ref{n}, a multiplicative structure is formed by the field $(\r_*, +_*,\cdot_* )$. Each element of the space $\r_*$ is referred to as a multiplicative number and is denoted by $a_*\in\r_*$, where $a_* = exp(a)$. For the sake of simplicity in notation, we will denote multiplicative numbers as $a\in \r_*$ instead of $a_*$ in the rest of the study. In addition, the unit elements of multiplicative addition and multiplication operations are $0_*=1$ and $1_*=e$, respectively.

Now some useful operations on the multiplicative space will be given. Multiplicative space is defined based on the absolute value multiplication operation.  Since distance is an additive change in Newtonian (additive) space, the absolute value is defined as additive.  However, since distance is a multiplicative change in multiplicative space, the multiplicative absolute value is as follows
\begin{equation*}
 \left \vert a \right \vert_*=\left\{ 
\begin{array}{ll}
a, & a\geq 0_*\quad \text{or}\quad a\in [ 1,\infty ) \\ 
-_*a, & a< 0_* \quad\text{or}\quad a \in (0,1).%
\end{array}%
\right.
\end{equation*} 
where $-_*a=1/a.$ Also in the multiplicative space we have,
\begin{equation*}
  a^{k*}=\underbrace{a\cdot_*a\cdot_*a\cdot_*...\cdot_*a}_{k-times}=e^{(\log a)^k} 
\end{equation*}
for $a\in\r_*$. Moreover for $k\in\r$, we have
\begin{equation*}
  a^{\frac{1}{2}*}=e^{(\log a)^\frac{1}{2}}=\sqrt[*]{a}. 
\end{equation*}
The inverses of multiplicative addition and multiplication in multiplicative space are as follows, respectively
$$
-_* a=1/a,\quad a^{-1_*}=e^{\frac{1}{\log a}}.
$$
We have the following formulas for $a,b\in\r_*$
\begin{equation}
    (a+_*b)^{2*}=a^{2*}+_*e^2\cdot_*a\cdot_*b+_*b^{2*},\label{kartop}
\end{equation}
\begin{equation}
    a^{2*}-_*b^{2*}=(a+_*b)\cdot_*(a-_*b).\label{karfar}
\end{equation}
The vector definition in $n$-dimensional multiplicative space $\r_*^n$ is given by
\begin{equation*}
 \r_*^n =\{(x_1 ,..., x_n) : x_i \in\r_*, i\in{1,2,3,...,n}\}.   
\end{equation*}
$\r_*^n$ is a vector space on $\r_*$ with the pair of operations
\begin{eqnarray*}
\u +_* \v &=&(u_1 +_* v_1 ,..., u_n +_* v_n)=(u_1  v_1 ,..., u_n  v_n), \\
a\cdot_*\u &=&(a \cdot_* u_1,..., a \cdot_* u_n)=(e^{\log a \log u_1},...,e^{\log a \log u_n} )=e^{\log a\log \u}, \quad a \in \r_*.
\end{eqnarray*}
where $\u,\v\in\r_*^n$ and $\log\u=(\log u_1,...,\log u_n).$

The symbol $\E^{2}_*$ will be used to denote the set $\r^2_*$ equipped with the multiplicative distance $d_*$.
Let $\u$ and $\v$ be any two multiplicative vectors in the multiplicative vector space $\r_*^n$. The multiplicative inner product of the $\u$ and $\v$ vectors is
\begin{equation*}
 \langle \u,\v\rangle_*=e^{\langle \log\u,\log\v\rangle}. 
\end{equation*}
Moreover, if the multiplicative vectors $\u$ and $\v$ are perpendicular to each other, a relation can be given as
\begin{equation*}
 \langle \u,\v\rangle_*=0_*. 
\end{equation*}
For a vector $\u\in\r_*^n$, the multiplicative norm of $\u$ is defined as follows,
\begin{equation*}
    \|\u\|_*=e^{\langle \log\x,\log\x\rangle^{\frac{1}{2}}}.
\end{equation*}
  The  multiplicative cross product of $\u$ and $\v$ in $\r_*^3$ is defined by
\begin{equation*}
 \u \times_* \v = (e^{\log u_2\log v_3 -\log u_3\log v_2},e^{\log u_3\log v_1 -\log u_1\log v_3}, e^{\log u_1\log v_2 -\log u_2\log v_1}) .
\end{equation*}
The multiplicative cross product exhibits the typical algebraic and geometric properties. For instance, $\mathbf{u} \times_* \mathbf{v}$ is multiplicative orthogonal to both $\mathbf{u}$ and $\mathbf{v}$. Additionally, $\mathbf{u} \times_* \mathbf{v = 0_*}$ if and only if $\mathbf{u}$ and $\mathbf{v}$ are multiplicative collinear. For example, let's consider the multiplicative vectors $\u=(e^5,e^3,-_*e^2)$ and $\v=(e^4,e^2,e^{13})$. Here can easily be calculated to be $\u\times_*\v$. These $\u$, $\v$ and $\u\times_*\v$ multiplicative vectors are formed into a multiplicative orthogonal system to each other  \cite{svetlin}. In Fig. \ref{fig1}, we present the graph of the multiplicative orthogonal system.
\begin{figure}[hbtp]
\begin{center}
\includegraphics[width=.28\textwidth]{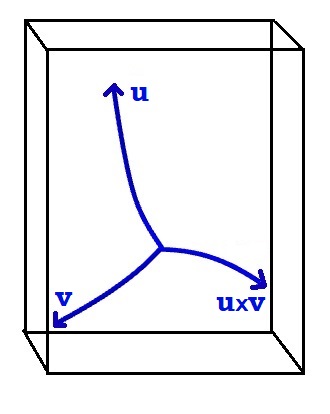}\\
\end{center}
\caption{\text{Multiplicative orthogonal system} \label{fig1}}
\end{figure}

Let $\u$ and $\v$ represent two unit direction multiplicative vectors in the multiplicative vector space. Let's denote by $\theta$ the multiplicative angle between the multiplicative unit vectors $\u$ and $\v$, so $\theta$ as follows
\begin{equation*}
    \theta=\arccos_*(e^{\langle \log\u,\log\v\rangle}).
\end{equation*}
For $\theta\in\r_*$, the definitions multiplicative sine, multiplicative cosine, multiplicative tangent and multiplicative cotangent are as follows
\begin{eqnarray*}
\sin_*\theta&=&e^{\sin\log\theta},\quad\quad\cos_*\theta=e^{\cos\log\theta},\\
\tan_*\theta&=&e^{\tan\log\theta},\quad\quad\cot_*\theta=e^{\cot\log\theta}.
\end{eqnarray*}
In addition, multiplicative trigonometric functions provide multiplicative trigonometric relations in parallel with classical trigonometric relations. For example, there is the equality $\sin_*^{2*}\theta+\cos_*^{2*}\theta=1_*$. For other relations, see \cite{svetlin}.

Let the function $f$ be given in multiplicative space $\r_*$, where $x\in I\subset\r_*$. The multiplicative derivative of $f$ is denoted by $f^*(x)$ and is as follows
\begin{eqnarray*}
f^*(x)&=&\lim_{h\rightarrow 0_*}(f(x+_*h)-_*f(x))/_*h\\
&=&\lim_{h\rightarrow 1}\left(\frac{f(xh)}{f(x)}\right)^{1/\log h}\\
&=&\lim_{h\rightarrow 1}e^{\frac{\log\frac{f(xh)}{f(x)}}{\log h}}.
\end{eqnarray*}
If the L' Hospital's rule applies here, we get
$$
f^*(x)=e^{\frac{xf'(x)}{f(x)}}.
$$
In addition, if the function $f$ is differentiable in the multiplicative sense and continuous, it is called * (multiplicative) differentiable function. It also satisfies the multiplicative derivative Leibniz and the chain rule (see details \cite{svetlin}).

The definition of the multiplicative integral is given as the inverse operator of the multiplicative derivative. The  multiplicative indefinite integral of the function $f(x)$ is defined by 
\begin{equation*}
 \int_* f(x)\cdot_*d_*x=e^{\int\frac{1}{x}\log f(x)dx},\quad x\in\r_*.
\end{equation*}
\section{Differential Geometry of Curves in Multiplicative Space} \label{sec2}
Let the $\x$ multiplicative vector valued function defined in the open interval $I\in\r_*$ be given $ \x:I\in\r_*\rightarrow\E_*^3,
s\mapsto\x(s)=(x_1(s),x_2(s),x_3(s)).$ The multiplicative component functions of $\x$ can be $k$-times continuously multiplicative differentiable. In this case, $\x$ is called of class $C_*^k$ ($k \geq 1_*$). Withal a differentiable $\x$ multiplicative is called a curve in multiplicative Euclidean space $\E_*^3$. In particular, a parametric multiplicative curve $\x$  is regular if and only if $\| \x^*(s) \|_* \neq 0_*$ for any $s \in I$. The expression $\x^*(s)$ gives us the multiplicative velocity function of $\x$. In addition, if $\| \x^*(s) \|_*=1$ for every $s\in\r_*$, then $\x$ is a multiplicative naturally parametrized curve. The multiplicative Frenet trihedron of a multiplicative  naturally parameterized curve $\x(s)$ are
$$\t(s)=\x^*(s),\quad \quad  \n (s)= \x^{**}(s)/_*\| \x^{**}(s)\|_*,\quad  \quad\b (s)= \t (s) \times_* \n (s).$$
\quad The vector field $\t(s)$ (resp. $\n(s)$ and $\b(s)$) along $\x(s)$ is said to be {\it multiplicative tangent} (resp. {\it multiplicative principal normal} and {\it multiplicative binormal}). It is direct to prove that $\{\t (s), \n (s), \b (s) \}$ is mutually multiplicative orthogonal and $\n (s) \times_* \b (s) =\t (s) $ and $\b (s) \times_* \t (s) =\n (s) $. We also point out that the arc length parameter and multiplicative Frenet frame are independent from the choice of multiplicative parametrization \cite{svetlin2}. 

The multiplicative Frenet formulae of $\x$ are given by
\begin{eqnarray*}
    \t^*&=&\kappa \cdot _* \n=e^{\log\kappa\log\n},\\
    \n^*&=&-_*\kappa \cdot_* \t +_* \tau \cdot_* \b=e^{-\log\kappa\log\t+\log\tau\log\b},\\
    \b^*&=&-_* \tau \cdot_* \n=e^{-\log\tau\log\n},
\end{eqnarray*}
where $\kappa=\kappa(s)$ and $\tau=\tau(s)$ are the curvature and the torsion functions of $\x$, calculated by
\begin{eqnarray*}
    \kappa(s)&=&\| \x^{**}(s)\|_*=e^{\langle \log\x^{**},\log\x^{**}\rangle^{\frac{1}{2}}},\\
    \tau(s)&=&\langle \n^*(s),\b(s) \rangle_* =e^{\langle \log\n^*(s),\log\b(s)\rangle}.
\end{eqnarray*}
Has A. et al. conducted a study on helices in multiplicative space and derived the general multiplicative helices equation as \cite{aykut}
\begin{equation}
 \tau/_*\kappa=c,\quad c\in\r_*. \label{hel1}  
\end{equation}
Furthermore, in the same study, the multiplicative slant helix equation is provided as \cite{aykut}
\begin{equation}
\sigma=\left(\kappa^{2*}/_*(\kappa^{2*}+_*\tau^{2*})^{\frac{3}{2}*}\right)\cdot_*(\tau/_*\kappa)^*  \label{hel2} 
\end{equation}
where $\sigma=c,\quad c\in\r_*$.
Evren M.E. et al examined spherical curves and obtained the following characterization \cite{evren}
\begin{equation*}
r=\left ( (e^{-1}/_*\kappa(s))^{2_*}+_* ((e^{-1}/_*\kappa(s))^*/_*\tau(s)))^{2_*}) \right )^{\frac{1}{2}_*}
\end{equation*}
or equivalently,
\begin{equation}
r^{2*}=p^{2*}+_*(p^*\cdot_*q)^{2*}\label{kur1}
\end{equation}
where $p=e/_*\kappa(s)$ and $q=e/_*\tau(s)$. Also multiplicative Frenet curve $\x$ is a multiplicative spherical curve if and only if 
\begin{equation}
(p^*\cdot_*q)^{*}+_*p/_*q=0_*.\label{kur2}
\end{equation}
Again in the study, Evren M. E. et al. are studied multiplicative rectifying curves and obtained the following characterization
\begin{equation*}
 \tau^{\frac{1}{\log\kappa}}=e^{\log a\log s+\log b} 
\end{equation*}
or multiplicative arguments
\begin{equation}
 \tau/_*\kappa=a\cdot_*s+_*b \label{rec}  
\end{equation}
where $a,b\in\r_*$. Additionally, some basic geometric concepts have been obtained. Evren M. E. et al. introduced the multiplicative sphere and circle \cite{evren}. Additionally, Has and Yılmaz examined multiplicative planes and conics. In Fig. \ref{fig2} shows the multiplicative unit circle and sphere whose centers is the multiplicative origin $O_1=(0_*,0_*)$ and $O_2=(0_*,0_*,0_*)$, respectively.
\begin{figure}[hbtp]
\begin{center}
\includegraphics[width=.45\textwidth]{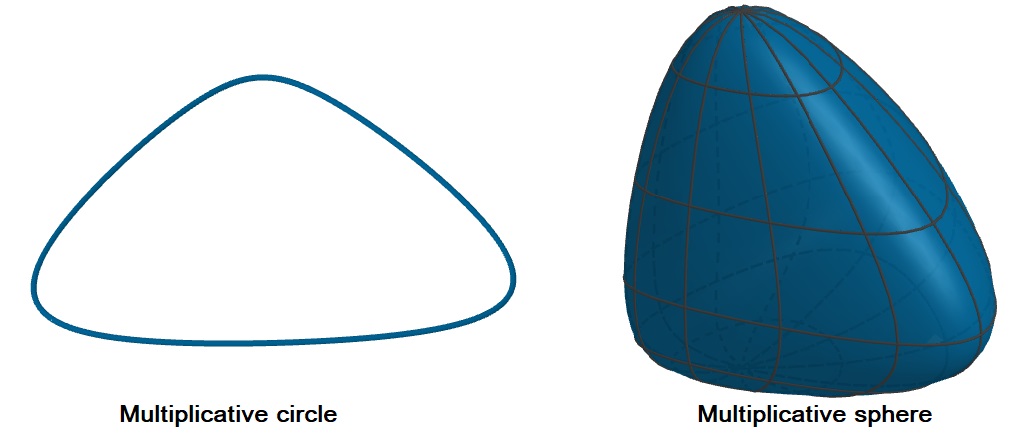}\\
\end{center}
\caption{\text{Multiplicative unit circle and sphere} \label{fig2}}
\end{figure}

We have also given the multiplicative plane spanned by multiplicative lines and whose equation is $e^3\cdot_*x-_*e^{-2}\cdot_*y+_*z-e^5=0_*$ in Fig. \ref{fig3}.
\begin{figure}[hbtp]
\begin{center}
\includegraphics[width=.45\textwidth]{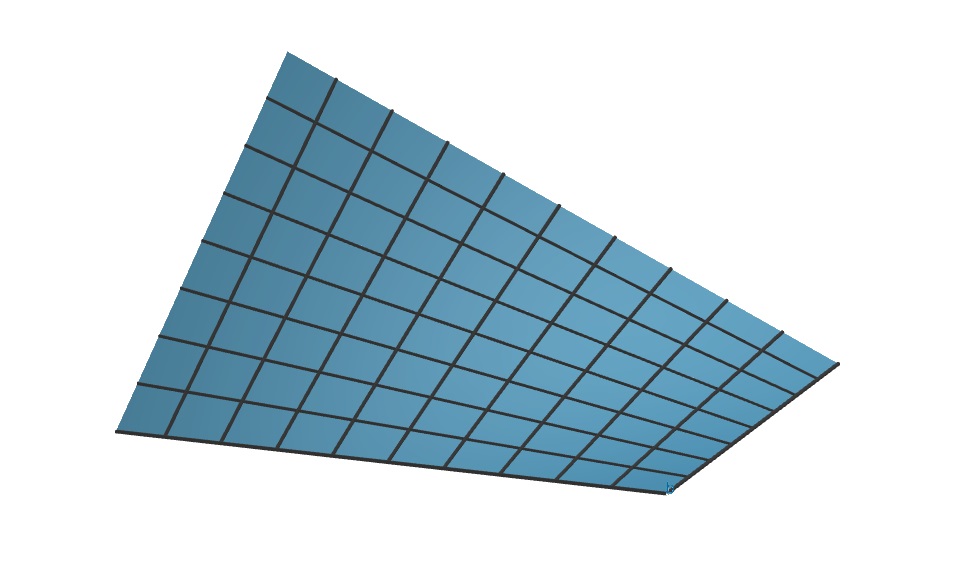}\\
\end{center}
\caption{\text{Multiplicative plane} \label{fig3}}
\end{figure}
\begin{conclusion}
Multiplicative analysis has some advantages over traditional (non-Newtonian) analysis. It is not possible to interpret the geometric interpretation of some exponential expressions with the help of traditional analysis. More clearly, consider the following subset of $\r^2$.
$$
C=\{(x,y) \in  \r^2: (\log x)^2+(\log y)^2=1, x,y>0 \}.
$$ 
We also can parameterize this set as $x(t)=e^{\cos (\log t)}$ and $y(t)=e^{\sin (\log 
t)}$, $t>0$. If we use the usual arithmetic operations, derivative and integral, then it would not be easy to understand what the set $C$ expresses geometrically. With or without the help of computer programs, we cannot even calculate its basic invariants, e.g., the arc length function $s(t)$ is given by a complicated integral
$$
s(t)=\int^t \frac{1}{u}\left ((\sin(\log u)e^{\cos (\log u)})^2+(\cos(\log u)e^{\sin (\log u)})^2 \right )^{1/2}du.
$$ 
However, applying the multiplicative tools, we see that $C$ is indeed a multiplicative circle parameterized by the multiplicative arc length whose center is $(1,1)$ and radius $e$, which is one of the simplest multiplicative curves. This is the reason why, in some cases, the multiplicative tools need to be applied instead of the usual ones \cite{evren}. \label{conc}
\end{conclusion}
\section{Multiplicative partner curves}
In this section, Bertrand and Mannheim curve pairs, which are popular curve pairs in differential geometry, will be characterized with the help of multiplicative arguments.
\subsection{Multiplicative Bertrand partner curves}
\begin{definition}
 Let $\x:I\in\r_*\rightarrow\E_*^3$ be the multiplicative curve in multiplicative Euclidean space $\E_*^3$. The multiplicative curve $\y$  with a multiplicative principal normal equal to the multiplicative principal normal of the $\x$ is called the Bertrand partner curve of the $\x$. In other words, let the multiplicative Frenet vectors of $\x$ and $\y$ be $\{\t,\n,\b\}$ and $\{\bar{\t},\bar{\n},\bar{\b}\}$, respectively. In case $\n=\bar{\n}$, $\y$ is called the multiplicative Bertrand partner curve of $\x$ and where the multiplicative angle between $\t$ and $\bar{\t}$ is $\theta.$
 \begin{figure}[hbtp]
\begin{center}
\includegraphics[width=.45\textwidth]{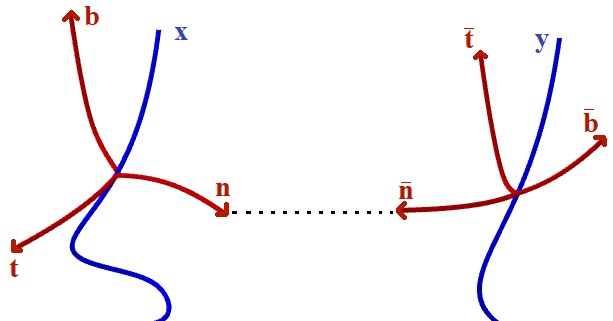}\\
\end{center}
\caption{\text{Multiplicative Bertrand partner curve and their multiplicative Frenet vectors.} \label{fig3}}
\end{figure}
\end{definition}
\begin{conclusion}
Let $\x$ and $\y$ be a pair of multiplicative Bertrand partner curves with multiplicative Frenet vectors $\{\t,\n,\b\}$ and $\{\bar{\t},\bar{\n},\bar{\b}\}$ in the multiplicative Euclidean space, respectively. The following equation exists between $\x$ and $\y$ with multiplicative arguments
\begin{equation}
 \y(s)=\x(s)+_*\lambda\cdot_*\n(s).\label{8}
\end{equation}
\end{conclusion}
\begin{remark}
Let $\x$ and $\y$ be a pair of multiplicative Bertrand partner curves with multiplicative Frenet vectors $\{\t,\n,\b\}$ and $\{\bar{\t},\bar{\n},\bar{\b}\}$ in the multiplicative Euclidean space, respectively. The following relations exist between the multiplicative Frenet vectors of $\x$ and $\y$ \label{remark}
\begin{eqnarray*}
\bar{\t}&=&\cos_*\theta\cdot_*\t-_*\sin_*\theta\cdot_*\b,\\
\bar{\b}&=&\sin_*\theta\cdot_*\t+_*\cos_*\theta\cdot_*\b.
\end{eqnarray*}
\end{remark}
\begin{proposition}
Let $\x$ and $\y$ form a multiplicative Bertrand partner curve in the multiplicative Euclidean space. In this case, $\lambda$ in Eq. \eqref{8} is a multiplicative constant. 
\begin{proof}
 Taking a multiplicative differentiation in Eq. \eqref{8}, we obtain
 \begin{equation*}
  \y^*=\x^*+_*\lambda^*\cdot_*\n+_*\lambda\cdot_*(-_*\kappa\cdot_*\t+_*\tau\cdot_*\b)   
 \end{equation*}
 and so
  \begin{equation}
  \bar{\t}=(e-_*\lambda\cdot_*\kappa)\cdot_*\t+_*\lambda^*\cdot_*\n+_*\lambda\cdot_*\tau\cdot_*\b. \label{9}
 \end{equation}
Since $\n$ and $\bar{\n}$ are multiplicative linear dependent, from Eq. \eqref{9}, we get
\begin{equation*}
    e^{\langle\log\bar{\t},\log\bar{\n}\rangle}=(e-_*\lambda\cdot_*\kappa)\cdot_*e^{\langle\log\bar{\n},\log\t\rangle}+_*\lambda^*\cdot_*e^{\langle\log\bar{\n},\log\n\rangle}+_*\lambda\cdot_*\tau\cdot_*e^{\langle\log\bar{\n},\log\b\rangle}
\end{equation*}
and
\begin{equation*}
    e^{0}=e^{\log\frac{e}{\lambda}\log\kappa\log e^0+\log\lambda^*+\log\lambda\log\tau\log e^0}.
\end{equation*}
Finally, let this equation be arranged in the light of multiplicative arguments, as follow
\begin{equation*}
    0_*=\lambda^*.
\end{equation*}
Take the multiplicative integral of both sides
\begin{eqnarray*}
    \int_*\lambda^*\cdot_*d_*s&=&\int_*0_*\cdot_*d_*s\\
    e^{\int\frac{1}{s}\log e^{\frac{s\lambda'}{\lambda}}ds}&=&e^{\int\frac{1}{s}\log e^{0}ds}\\
    e^{\int\frac{\lambda'}{\lambda}ds}&=&e^c\\
    e^{\log\lambda}&=&e^c\\
    \lambda&=&e^c, c\in\r.
\end{eqnarray*}
Thus, it is proven that $\lambda\in\r_*$.
\end{proof}
\end{proposition}
\begin{proposition}
Let $\x$ and $\y$ be a pair of multiplicative Bertrand partner curves with multiplicative Frenet apparatus $\{\t,\n,\b\,\kappa,\tau\}$ and $\{\bar{\t},\bar{\n},\bar{\b},\bar{\kappa},\bar{\tau}\}$ in the multiplicative Euclidean space, respectively. In this case, the following equations exist
\begin{eqnarray*}
  \bar{\kappa}&=&e^{(c\log\kappa-\sin^2\log\theta)/(c-c^2\log\kappa)}\\  
  \bar{\tau}&=&e^{\sin^2\log\theta/c^2\log\tau}
\end{eqnarray*}
where $c=\log\lambda$.
\end{proposition}
\begin{proof}
We use take Remark \ref{remark} in Eq \eqref{9}, we get
\begin{equation}
  e^{\langle\log\bar{\t},\log\t\rangle}=(e-_*\lambda\cdot_*\kappa)\cdot_*e^{\langle\log\t,\log\t\rangle}+_*\lambda^*\cdot_*e^{\langle\log\n,\log\t\rangle}+_*\lambda\cdot_*\tau\cdot_*e^{\langle\log\b,\log\t\rangle} \label{10}
 \end{equation}
 and
 \begin{equation}
  e^{\langle\log\bar{\t},\log\b\rangle}=(e-_*\lambda\cdot_*\kappa)\cdot_*e^{\langle\log\t,\log\b\rangle}+_*\lambda^*\cdot_*e^{\langle\log\n,\log\b\rangle}+_*\lambda\cdot_*\tau\cdot_*e^{\langle\log\b,\log\b\rangle}. \label{11}
  \end{equation}
  If the necessary adjustments are made in Eqs. \eqref{10} and \eqref{11}, the following results are obtained
  \begin{eqnarray}
   e-_*\lambda\cdot_*\kappa&=&e^{\cos\log\theta} \quad \text{or}\quad  e-_*\lambda\cdot_*\kappa=\cos_*\theta \label{12}\\
  \lambda\cdot_*\tau&=&e^{-\sin\log\theta} \quad \text{or}\quad -_*\lambda\cdot_*\tau=\sin_*\theta.\label{13}
  \end{eqnarray}
  On the other hand, starting from Eq. \eqref{9}, we give
  \begin{equation}
 \x(s)=\y(s)-_*\lambda\cdot_*\bar{\n}(s).\label{14}
\end{equation}
we taking multiplicative differentiation in Eq. \eqref{14}, as follow
\begin{equation*}
  \x^*=\y^*-_*\lambda\cdot_*(-_*\kappa\cdot_*\t+_*\tau\cdot_*\b)   
 \end{equation*}
 and so
  \begin{equation}
  \t=(e+_*\lambda\cdot_*\bar{\kappa})\cdot_*\bar{\t}-_*\lambda\cdot_*\bar{\tau}\cdot_*\bar{\b}. \label{15}
 \end{equation}
 Since the positions of the multiplicative curves change, the constant multiplicative angle between $\t$ and $\bar{\t}$ becomes $-_*\theta$. Moreover, from the equations $\cos_*(-_*\theta)=e^{\cos\log(\frac{1}{\theta})}=\cos_*\theta$ and $\sin_*(-_*\theta)=e^{\sin\log(\frac{1}{\theta})}=-_*\sin_*\theta$, it is easily seen that $\cos_*\theta$ is a multiplicative even function and $\sin_*\theta$ is a multiplicative odd function. We use take Remark \ref{remark} in Eq \eqref{15}, we get
\begin{equation}
e^{\langle\log\t,\log\bar{\t}\rangle}=(e+_*\lambda\cdot_*\bar{\kappa})\cdot_*e^{\langle\log\bar{\t},\log\bar{\t}\rangle}-_*\lambda\cdot_*\bar{\tau}\cdot_*e^{\langle\log\bar{\b},\log\bar{\t}\rangle} \label{16} 
 \end{equation}
 and
\begin{equation}
e^{\langle\log\t,\log\bar{\t}\rangle}=(e+_*\lambda\cdot_*\bar{\kappa})\cdot_*e^{\langle\log\bar{\t},\log\bar{\t}\rangle}-_*\lambda\cdot_*\bar{\tau}\cdot_*e^{\langle\log\bar{\b},\log\bar{\t}\rangle}. \label{17}  
 \end{equation}
 From Eqs. \eqref{16} and \eqref{17}, we obtain
 \begin{eqnarray}
e+_*\lambda\cdot_*\bar{\kappa}&=&e^{\cos\log\frac{1}{\theta}} \quad \text{or}\quad  e+_*\lambda\cdot_*\bar{\kappa}=\cos_*\theta \label{18}\\
-_*\lambda\cdot_*\bar{\tau}&=&e^{\sin\log\frac{1}{\theta}} \quad \text{or}\quad -_*\lambda\cdot_*\bar{\tau}=\sin_*\theta.\label{19}
  \end{eqnarray}
Next, we consider Eqs. \eqref{12} and \eqref{18}, we have
\begin{equation}
\cos_*^{2*}\theta=(e-_*\lambda\cdot_*\kappa)\cdot_*(e+_*\lambda\cdot_*\bar{\kappa}).
\end{equation}
From above equation as follow
\begin{eqnarray*}
 \bar{\kappa}&=&e^{[\log e^{(\log e^{\cos\log\theta})^2/(\log e-\log\lambda\log\kappa)-\log e}]/\log\lambda} \\
 &=&e^{[\cos^2\log\theta/(1-\log\lambda\log\kappa)-1]\log\lambda}.
\end{eqnarray*}
Consider that $c=\log\lambda$ in the above equation and make the necessary multiplicative adjustments, we obtain
\begin{equation*}
  \bar{\kappa}=e^{(c\log\kappa-\sin^2\log\theta)/(c-c^2\log\kappa)}. 
\end{equation*}
Similarly, Eqs \eqref{13} and \eqref{18} can be written as follows
\begin{equation*}
\sin_*^{2*}\theta=\lambda^{2*}\cdot_*\tau\cdot_*\bar{\tau}    
\end{equation*}
and so
\begin{eqnarray*}
\bar{\tau}&=&e^{\log e^{(\log e^{\sin\log\theta})^2}/\log e^{(\log \lambda)^2}\log\tau}\\
&=&e^{\sin^2\log\theta/\log^2\lambda\log\tau}.
\end{eqnarray*}
Let's take into account the equality $c=\log\lambda$ here, so this happens
\begin{eqnarray*}
 \bar{\tau}=e^{\sin^2\log\theta/c^2\log\tau}.   
\end{eqnarray*}
\end{proof}
\begin{theorem}
 Let $\x$ and $\y$ be a pair of multiplicative Bertrand partner curves. The following equation is satisfied, with the curvatures of $\x$ being $\kappa$ and $\tau$ respectively,
 \begin{equation*}
  \kappa^{\log\lambda}+_*\tau^{\log\mu}=e.  
 \end{equation*}
\end{theorem}
\begin{proof}
From Eq. \eqref{9} we know that
\begin{equation*}
  \bar{\t}=(e-_*\lambda\cdot_*\kappa)\cdot_*\t+_*\lambda^*\cdot_*\n+_*\lambda\cdot_*\tau\cdot_*\b.
 \end{equation*}
 Let Remark \ref{remark} be used here, then the following can be written
 \begin{eqnarray*}
 e^{\frac{1-\log\lambda\log\kappa}{\cos\log\theta}}&=&e^{\frac{\log\lambda\log\tau}{\sin\log\theta}}
\end{eqnarray*}
 If the above equation is rearranged, the following result is obtained
 \begin{eqnarray*}
e^{\log\lambda\log\kappa+\log\lambda\log\tau \cot\log\theta}&=&\lambda\cdot_*\kappa+_*\lambda\cdot_*\cot_*\theta\cdot_*\tau\\
&=&e.
 \end{eqnarray*}
Since $\theta$ angle is a multiplicative constant angle and $\lambda$ is a multiplicative constant, we can say for $\mu\in\r_*$, $\lambda\cdot_*\cot_*\theta=\mu$. Then we express the above equation as follows
\begin{equation*}
\lambda\cdot_*\kappa+_*\mu\cdot_*\tau=e    
\end{equation*}
or equivalently
\begin{equation*}
  \kappa^{\log\lambda}+_*\tau^{\log\mu}=e.  
 \end{equation*}
\end{proof}
\begin{example}
Let $\x:I\subset\mathbb{R}_*\rightarrow E^{3}_*$ be a multiplicative naturally parameterized curve in $\mathbb{R}^{3}_*$ parameterized by
\begin{equation*}
\x(s)=\left(-_*e/_*e^{\sqrt{2}}\cdot_*\sin_*s,e/_*e^{\sqrt{2}}\cdot_*\cos_*s,e/_*e^{\sqrt{2}}\cdot_*e^s\right).
\end{equation*}
The multiplicative Bertrand curve pair of $\x(s)$ is obtained as follows,
\begin{equation*}
\y(s)=\left(e/_*e^{\sqrt{2}}\cdot_*\cos_*s,-_*e/_*e^{\sqrt{2}}\cdot_*\sin_*s,e/_*e^{\sqrt{2}}\cdot_*e^s\right).
\end{equation*}
 In Fig. \ref{fig5} we shows the multiplicative Bertrand partner curves $\x(s)$ and $\y(s)$.
\begin{figure}[hbtp]
\begin{center}
\includegraphics[width=.24\textwidth]{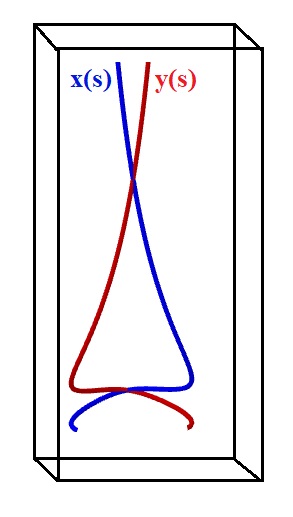}\\
\end{center}
\caption{\text{Multiplicative Bertrand partner curve.} \label{fig5}}
\end{figure}
\end{example}
\subsection{Multiplicative Mannheim partner curves}
\begin{definition}
 Let $\x:I\in\r_*\rightarrow\E_*^3$ be the multiplicative curve in multiplicative Euclidean space $\E_*^3$. The multiplicative curve $\y$  with a multiplicative binormal equal to the multiplicative principal normal of the $\x$ is called the multiplicative Manneim partner curve of the $\x$. In other words, let the multiplicative Frenet vectors of $\x$ and $\y$ be $\{\t,\n,\b\}$ and $\{\bar{\t},\bar{\n},\bar{\b}\}$, respectively. In case $\n=\bar{\b}$, $\y$ is called the multiplicative Mannheim partner curve of $\x$ and where the multiplicative angle between $\t$ and $\bar{\t}$ is $\theta.$
  \begin{figure}[hbtp]
\begin{center}
\includegraphics[width=.45\textwidth]{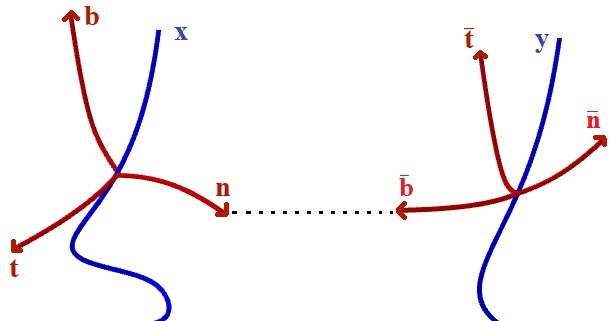}\\
\end{center}
\caption{\text{Multiplicative Mannheim partner curve and their multiplicative Frenet vectors.} \label{fig3}}
\end{figure}
\end{definition}
\begin{conclusion}
Let $\x$ and $\y$ be a pair of multiplicative Mannheim partner curves with multiplicative Frenet vectors $\{\t,\n,\b\}$ and $\{\bar{\t},\bar{\n},\bar{\b}\}$ in the multiplicative Euclidean space, respectively. The following equation exists between $\x$ and $\y$ with multiplicative arguments
\begin{equation}
 \y(s)=\x(s)+_*\lambda\cdot_*\n(s)\label{21}
 \end{equation}
 and
 \begin{equation}
 \x(s)=\y(s)+_*\mu\cdot_*\bar{\b}(s).\label{22}
\end{equation}
\end{conclusion}
\begin{remark}
Let $\x$ and $\y$ be a pair of multiplicative Mannheim partner curves with multiplicative Frenet vectors $\{\t,\n,\b\}$ and $\{\bar{\t},\bar{\n},\bar{\b}\}$ in the multiplicative Euclidean space, respectively. The following relations exist between the multiplicative Frenet vectors of $\x$ and $\y$ \label{remark1}
\begin{eqnarray*}
\bar{\t}&=&\cos_*\theta\cdot_*\t-_*\sin_*\theta\cdot_*\b,\\
\bar{\n}&=&\sin_*\theta\cdot_*\t+_*\cos_*\theta\cdot_*\b.
\end{eqnarray*}
\end{remark}
\begin{proposition}
Let $\x$ and $\y$ form a multiplicative Bertrand partner curve in the multiplicative Euclidean space. In this case, $\lambda$ and $\mu$ in Eqs. \eqref{21} and \eqref{22}, respectively is a multiplicative constant.
\end{proposition}
\begin{proof}
Taking a multiplicative differentiation in Eq. \eqref{21}, we obtain
\begin{equation*}
  \y^*=\x^*+_*\lambda^*\cdot_*\n+_*\lambda\cdot_*(-_*\kappa\cdot_*\t+_*\tau\cdot_*\b)   
 \end{equation*}
 and so
  \begin{equation}
  \bar{\t}=(e-_*\lambda\cdot_*\kappa)\cdot_*\t+_*\lambda^*\cdot_*\n+_*\lambda\cdot_*\tau\cdot_*\b. \label{23}
 \end{equation}
 Since vector $\y^*$ is multiplicative orthogonal to multiplicative vector $\bar{\b}$ and $\bar{\b}$ is multiplicative linearly dependent on $\n$, as follow
\begin{equation*}
    0_*=\lambda^*.
\end{equation*}
Take the multiplicative integral of both sides
\begin{eqnarray*}
    \int_*\lambda^*\cdot_*d_*s&=&\int_*0_*\cdot_*d_*s\\
    e^{\int\frac{1}{s}\log e^{\frac{s\lambda'}{\lambda}}ds}&=&e^{\int\frac{1}{s}\log e^{0}ds}\\
    e^{\int\frac{\lambda'}{\lambda}ds}&=&e^c\\
    e^{\log\lambda}&=&e^c\\
    \lambda&=&e^c, c\in\r.
\end{eqnarray*}
Similarly, it is shown that $\mu$ is a multiplicative constant.
\end{proof}
\begin{proposition}
Let $\x$ and $\y$ be a pair of multiplicative Mannheim partner curves with multiplicative Frenet apparatus $\{\t,\n,\b\,\kappa,\tau\}$ and $\{\bar{\t},\bar{\n},\bar{\b},\bar{\kappa},\bar{\tau}\}$ in the multiplicative Euclidean space, respectively. In this case, the following equations exist
\begin{eqnarray*}
  \bar{\kappa}&=&e^{(\log\kappa)^2\cos^2\log\theta+(\log\tau)^2\sin^2\log\theta^\frac{1}{2}}\\  
  \bar{\tau}&=&e^{\log\kappa\sin\log\theta-\log\tau\cos\log\theta}
\end{eqnarray*}
where $c=\log\lambda$.
\end{proposition}
\begin{proof}
 Suppose that $\x:I\in\r_*\rightarrow\E_*^3$ be a multiplicative curve. By definition of multiplicative curvature in \cite{svetlin2}, we have 
 \begin{equation}
\bar{\kappa}=e^{\langle\log\y^{**},\log\y^{**}\rangle^{\frac{1}{2}}}\quad\text{and}\quad\bar{\tau}=e^{\langle\log\bar{\n}^*,\log\bar{\b}\rangle}   \label{24}
 \end{equation}
 If the Remark \ref{remark1} is used in the first equation given above, we get
 \begin{eqnarray*}
\bar{\kappa}&=&e^{\langle\log\bar{\t}^{*},\log\bar{\t}^{*}\rangle^{\frac{1}{2}}}\\
&=&e^{((\log\kappa)^2\cos^2\log\theta+(\log\tau)^2\sin^2\log\theta)^\frac{1}{2}}
 \end{eqnarray*}
or with multiplicative arguments
\begin{equation*}
 \bar{\kappa}=(\kappa^{2*}\cdot_*\cos^{2*}_*\theta+_* \tau^{2*}\cdot_*\sin^{2*}_*\theta)^{\frac{1}{2}*}.  
\end{equation*}
On the other hand if the Remark \ref{remark1} is used in Eq. \eqref{24}, we get
\begin{eqnarray*}
\bar{\tau}&=&e^{(\log\kappa\sin\log\theta-\log\tau\cos\log\theta)\langle\log\n,\log\bar{\b}\rangle}\\
&=&e^{\log\kappa\sin\log\theta-\log\tau\cos\log\theta}
\end{eqnarray*}
or with multiplicative arguments
\begin{equation*}
 \bar{\tau}=\kappa\cdot_*\sin_*\theta-_* \tau\cdot_*\cos_*\theta.  
\end{equation*}
\end{proof}
\begin{theorem}
In order for $\x:I\in\r_*\rightarrow\E_*^3$, whose curvatures are $\kappa$ and $\tau$, to be a Mannheim curve, the following equations must be satisfied for $\lambda\in\r_*$   
\begin{equation*}
\tau^*/_*\kappa^*=-_*\cot_*\theta.
\end{equation*}
and
\begin{equation*}
 \kappa=e^{\log\lambda[(\log\kappa)^2+(\log\tau)^2]}.
\end{equation*}
\end{theorem}
\begin{proof}
Let considering  calculate the multiplicative differentiating of Eq. \eqref{23} with respect to $s$ and considering $\lambda$ is multiplicative constant, we get 
\begin{equation}
\bar{\kappa}\cdot_*\bar{\n}=-_*\lambda\cdot_*\kappa^*\cdot_*\t+_*(\kappa-_*\lambda\cdot_*\kappa^{2*}-_*\lambda\cdot_*\tau^{2*})\cdot_*\n+_*\lambda\cdot_*\tau^*\cdot_*\b. \label{25}  
\end{equation}
According to Remark \ref{remark1} we can give
\begin{eqnarray*}
 \bar{\kappa}\cdot_*\sin_*\theta&=&-_*\lambda\cdot_*\kappa^*,\\
\bar{\kappa}\cdot_*\cos_*\theta&=&\lambda\cdot_*\tau^*.
\end{eqnarray*}
From the above equations we easily obtain
\begin{eqnarray*}
\tau^*/_*\kappa^*&=&-_*\cot_*\theta.
\end{eqnarray*}
On the other hand, considering that $\n$ and $\bar{\b}$ are multiplicative linear dependencies in Eq. \eqref{25}, we give
\begin{eqnarray*}
\kappa-_*\lambda\cdot_*\kappa^{2*}-_*\lambda\cdot_*\tau^{2*}&=&0_*
\end{eqnarray*}
and so 
\begin{equation*}
\kappa=^{\log\lambda[(\log\kappa)^2+(\log\tau)^2]}.
\end{equation*}
\end{proof}

\begin{example}
Let $\x:I\subset\mathbb{R}_*\rightarrow E^{3}_*$ be a multiplicative naturally parameterized curve in $\mathbb{R}^{3}_*$ parameterized by
\begin{equation*}
\x(s)=\left(-_*e^8/_*e^{5}\cdot_*\cos_*s,e^8/_*e^{5}\cdot_*\sin_*s,e^4/_*e^{5}\cdot_*e^s\right).
\end{equation*}
The multiplicative Mannheim curve pair of $\x(s)$ is obtained as follows,
\begin{equation*}
\y(s)=\left(-_*e^8/_*e^{5}\cdot_*(\sin_*s+_*\cos_*s),e^8/_*e^{5}\cdot_*(\sin_*s+_*\cos_*s),e^4/_*e^{5}\cdot_*e^s\right).
\end{equation*}
 In Fig. \ref{fig7} we shows the multiplicative Mannheim partner curves $\x(s)$ and $\y(s)$.
\begin{figure}[hbtp]
\begin{center}
\includegraphics[width=.34\textwidth]{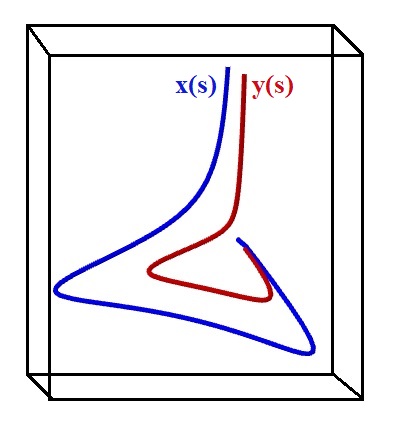}\\
\end{center}
\caption{\text{Multiplicative Mannheim partner curve.} \label{fig7}}
\end{figure}
\end{example}

\section{Conclusion}\label{sec13}
In this article, various curve pairs were scrutinized employing multiplicative arguments. The primary rationale behind incorporating multiplicative analysis in this investigation stems from its capability to resolve certain problems that are beyond the scope of traditional (Newtonian) analysis, as elucidated in Conclusion \ref{conc}. The utilization of diverse analytical methodologies, notably non-Newtonian analyses, holds significant importance in the realm of differential geometry. Such approaches not only offer alternative perspectives where conventional methods fall short but also facilitate solutions that are often more straightforward to attain. Therewithal the multiplicative space is produced with the help of the exponential function, it functions in the first quadrant of the traditional coordinate system. This transformation entails shifting the range $(-\infty,0)$ to $(0,1)$ and $(0,+\infty)$ to $(1,+\infty)$ effectively confining the multiplicative Euclidean space within the first quadrant region. Consequently, subjects analyzed using multiplicative arguments find themselves operating within a more constrained domain, facilitating a more streamlined examination process. Moreover, the proportional nature of measurements in multiplicative space enables a more efficient modeling of exponentially changing phenomena. Problems characterized by rapid exponential changes can be more effectively addressed with multiplicative arguments, allowing for a more numerical approach towards attaining real solutions compared to traditional methods.However, it's worth noting that compressing multiplicative space into the first quadrant may pose disadvantages for certain problem structures. Depending on the nature of the problems under scrutiny, this aspect warrants careful consideration to ensure comprehensive analysis and accurate outcomes.

\backmatter

\bmhead{Acknowledgements}

\section*{Declarations}

\begin{itemize}
\item \textbf{Funding:} The authors declare that no funds, grants, or other support were received during the preparation of this manuscript.
\item \textbf{Conflict of interest:}
 The authors declare that there is no conflict of interest.
\item \textbf{Ethical Approval:} It is invalid.
\item \textbf{Data availability:} It is invalid.
\end{itemize}


\begin{thebibliography}{99}

\bibitem{carmo}Carmo M. P. D., Differential Geometry of Curves and Surfaces, Prentice Hall Inc., New Jersey, 1976.

\bibitem{bertrand} Bertrand, J. , M\'{e}moire sur la th\'{e}orie des courbes \`{a} double courbure. Comptes Rendus 36; Journal de Math\'{e}matiques Pures et Appliqu\'{e}es 15, 332--350, 1850.

\bibitem{mannheim} A. Mannheim, De l'emploi de la Courbe Représentative de la Surface des Normales Principales d'une Courbe Gauche Pour la Démonstration de Propriétés Relatives à Cette Courbure. Comptes Rendus des Séances de l'Académie des Sciences, Paris C.R., 86 (1878), 1254–1256. 

\bibitem{boyer}Boyer C., A History of Mathematics, Wiley, New York, 1968.

\bibitem{hilal}Şenyurt S., Ayvacı K.H., Canlı D., 
Smarandache Curves According to Flc-frame in Euclidean 3-space, Fundamentals of Contemporary Mathematical Sciences, 4(1) (2023), 16-30.

\bibitem{sernesi}Sernesi E., Gaussian maps of plane curves with nine singular points, Ann Univ Ferrara, 63 (2017), 201–210.

\bibitem{takahashi}Takahashi M., Equi-affine plane curves with singular points, J. Geom. 113 (2022), 16.

\bibitem{aykut1}Has A., Yılmaz B., Baleanu D., On the Geometric and Physical Properties of Conformable Derivative, Mathematical Sciences and Applications E-Notes, 12(2) (2024), 60-70.

\bibitem{aykut2}Has A., Yılmaz B., Measurement and Calculation on Conformable Surfaces, Mediterr. J. Math., 20 (2023), 274. 

\bibitem{evren} Aydin M.E., Has A., Yilmaz B., A non-Newtonian approach in differential geometry of curves: multiplicative rectifying curves. ArXiv, (2023),  
https://doi.org/10.48550/arXiv.2307.16782

\bibitem{aykut3}Has A., Yılmaz B., Effect of fractional analysis on some special curves, Turkish Journal of Mathematics, 47(5) (2023), 1423-1436.

\bibitem{aykut4} Has A., Yılmaz, B., Ayvacı K.H., $C_\alpha$-
ruled surfaces respect to direction curve in fractional differential geometry, J. Geom., 115 (2024), 11.

\bibitem{zehra} Durmaz H., Özdemir Z., Şekerci Y., Fractional approach to evolution of the magnetic field lines near the magnetic null points, Physica Scripta, 99(2) (2024), 025239.

\bibitem{mert}Taşdemir M., Canfes E. Ö., Uzun B, On Caputo fractional Bertrand curves in $E^3$ and $E_1^3$, Filomat, 38(5) (2024), 1681–1702.

\bibitem{volterra} Volterra V., Hostinsky B., Operations Infinitesimales Lineares, Herman, Paris, 1938.

\bibitem{grossman} Grossman M., Katz R., Non-Newtonian Calculus, 1st ed., Lee Press, Pigeon Cove Massachussets, 1972.

\bibitem{grossman2} Grossman M., Bigeometric Calculus: A System with a Scale-Free Derivative, Archimedes Foundation, Massachusetts, 1983.

\bibitem{samuelson}Samuelson W.F. , Mark S.G., Managerial Economics, Wiley, New York, 2012.

\bibitem{afrouzi}Afrouzi H. H., Ahmadian M., Moshfegh A., Toghraie D., Javadzadegan A., Statistical analysis of pulsating non-Newtonian flow in a corrugated channel using Lattice-Boltzmann method. Physica A: Statistical Mechanics and its Applications, 535 (2019), 122486

\bibitem{rybaczuk} Rybaczuk M., Stoppel P., The fractal growth of fatigue defects in materials, International Journal of Fracture, 103 (2000), 71-94.

\bibitem{boruah} K. Boruah, B. Hazarika, Some Basic Properties of Bigeometric Calculus and its Applications in Numerical Analysis, Afrika Matematica, 32 (2021), 211-227.

\bibitem{bashirov} Bashirov A.E., Kurpınar E.M., Özyapıcı A., Multiplicative calculus and its applications. J. Math. Anal. Appl., 337 (2008), 36–48.

\bibitem{emrah1} Gulsen T., Yilmaz E., Goktas S., Multiplicative Dirac system. Kuwait J.Sci., 49(3) (2022), 1-11.

\bibitem{yusuf}Mısırlı, E., Gürefe, Y. Multiplicative Adams Bashforth–Moulton methods,  Numerical Algorithms, 57 (2011), 425–439.

\bibitem{svetlin} Georgiev S.G., Zennir K., Multiplicative Differential Calculus (1st ed.), Chapman and Hall/CRC., New York, 2022.

\bibitem{svetlin2}Georgiev S.G., Multiplicative Differential Geometry (1st ed.), Chapman and Hall/CRC., New York, 2022.

\bibitem{svetlin3} Georgiev S.G., Zennir K., Boukarou A., Multiplicative Analytic Geometry (1st ed.), Chapman and Hall/CRC., New York, 2022.


\bibitem{karacan} Nurkan S.K., Gurgil I., Karacan M.K., Vector properties of geometric calculus. Math. Meth. Appl. Sci., (2023), 1–20.

\bibitem{hasan}Es H., Plane kinematics in homothetic multiplicative calculus, Journal of Universal Mathematics, 7(1) (2024), 37-47.

\bibitem{hasan2}Es H., On the 1-parameter motions with multiplicative calculus, Journal of Science and Arts, 22(2) (2022), 395-406.

\bibitem{hasan3}Es H., On the 1-parameter motions with multiplicative calculus, Journal of Science and Arts, 22(2) (2022), 395-406.

\bibitem{aykut5} Has A., Yılmaz B., On non-Newtonian Helices in Multiplicative Euclidean Space $\E_*^3$, In Press. 

\bibitem{aykut6} Has A., Yılmaz B., A non-Newtonian Conics in Multiplicative Analytic Geometry, In Press.



\end{thebibliography}
\end{document}